\documentclass[11pt, a4paper]{amsart}

\usepackage[utf8]{inputenc}

\usepackage{graphicx}
\usepackage{caption}
\usepackage{amsmath}
\usepackage{amssymb}
\usepackage{amsbsy}
\usepackage{enumitem}
\usepackage{amsfonts}
\usepackage[english]{babel}
\usepackage{amsthm}

\usepackage[T1]{fontenc}
\usepackage{appendix}
\usepackage{graphicx}
\usepackage{tikz-cd}
\usepackage{cite}
\usepackage{pifont}
\usepackage{dsfont}
\usepackage{psfrag}

\usepackage[right=2.5cm, left=2.5cm, top=2.5cm, bottom=2.5cm]{geometry}

\title{Singularities admitting contracting automorphisms}

\author{Kémo Morvan}
\address{Université Paris Cité, Sorbonne Université, CNRS, IMJ-PRG, F-75013 Paris, France.}
\email{kmorvan@imj-prg.fr}

\subjclass{Primary 32S05, 37F99; Secondary 32B10, 14B05, 32A10.}

\keywords{Contracting automorphisms, quasi-homogeneous singularities.}

\theoremstyle{plain}

\newtheorem{mainthm}{Theorem}

\newtheorem{thm}{Theorem}[section]

\newtheorem{lem}[thm]{Lemma}
\newtheorem{prop}[thm]{Proposition}

\theoremstyle{definition}

\newtheorem{definition}[thm]{Definition}
\newtheorem{rems}[thm]{Remarks}
\newtheorem{rem}[thm]{Remark}
\newtheorem{ex}[thm]{Example}

\newcommand{\CC}{\mathbb{C}}

\newcommand{\NN}{\mathbb{N}}
\newcommand{\ZZ}{\mathbb{Z}}

\newcommand{\Ideal}{\mathcal{I}}
\newcommand{\Jdeal}{\mathcal{J}}
\newcommand{\Kdeal}{\mathcal{K}}
\newcommand{\Alp}{\boldsymbol{\alpha}}
\newcommand{\Bet}{\boldsymbol{\beta}}

\newcommand{\Ho}{\boldsymbol{H}}

\newcommand{\abs}[1]{\left|#1\right|}

\begin{document}
	
	\maketitle
	
	\begin{abstract}
		Given a complex analytic singularity $(X,0)$, we show that if there exists an automorphism $F : (X, 0) \to (X, 0)$ that is contracting, then $(X,0)$ is quasi-homogeneous.
	\end{abstract}
	
	\section{Introduction}\label{sect:Intro}
	\quad \quad Let $(X, 0)$ be a complex singularity, that is a germ of complex analytic space. One may wonder which geometrical conditions of $(X, 0)$ we can capture by studying its group of complex analytic automorphisms or endomorphisms. In fact, as Müller proved in 1987, the group of automorphisms of the singularity is always ``infinite dimensional'', in the sense that $(X,0)$ carries $m$ analytic vector fields that are in mutual involution and linearly independent for any $m \geq 1$ \cite{Muller}. We may wonder nevertheless if the existence of ``special'' endomorphisms would restrict the structure of the singularity. For example, Wahl proved that if there exists a non-invertible and finite self-map on a germ of complex surface, then up to a finite cover it is either smooth, simple elliptic or a cusp singularity \cite{Wahl}. 
	
	In this paper, we are interested in understanding the structure of a singularity admitting a contracting automorphism, see definition \ref{def:contr}. This question has previously been investigated by Favre and Ruggiero \cite{FavreRuggiero} in dimension $2$. In op. cit., they proved that a normal surface singularity can support a contracting automorphism if and only if it is quasi-homogeneous, see definition \ref{def:qhs}. A similar result was obtained by Camacho, Movasati and Scárdua \cite{CMS} when they studied Stein analytic spaces of dimension 2 admitting a $\mathbb{C}^*$-action with isolated singularities and at least one dicritical singularity.
	
	These results are inspired by the seminal work of Orlik and Wagreich \cite{OrlikWagreich}, where they proved that quasi-homogeneous singularities, in any dimension, are the only ones that admit a good $\mathbb{C}^*$-action. Such an action induces a family of contracting automorphisms on the singularity, but not all contracting automorphisms can be obtained in this way. In this paper, we generalize these results to contracting automorphisms on an arbitrary singularity:
	
	\begin{mainthm}
		Let $(X, 0)$ be a complex singularity and $f: (X,0) \to (X, 0)$ be a contracting automorphism. Then $(X, 0)$ is a quasi-homogeneous singularity.
	\end{mainthm}
	
	The results in \cite{CMS} and in \cite{FavreRuggiero} rely on the precise description of quasi-homogeneous surface singularities, and of the dual graph of their resolution, depicted in \cite{OrlikWagreich}. Such a description is not available in higher dimensions: to prove our main theorem, we go back to the original strategy of \cite{OrlikWagreich}, and adapt it to our setting. It consists in two steps: firstly, we find an embedding of the singularity $(X, 0)$ into $(\CC^d, 0)$ so that the contracting automorphism $f$ is the restriction of a contracting automorphism $F$ of the ambient space, see theorem \ref{thm:GoodEmb}. This allows to use Poincaré-Dulac normal forms up to holomorphic change of coordinates. Lastly, we show that any subvariety of $(\CC^d, 0)$ that is invariant by a contracting automorphism in Poincaré-Dulac normal form must be quasi-homogeneous, see theorem \ref{thm:InvId}.
	
	While any contracting automorphism $f$ of the singularity $(X, 0)$ is the restriction of an endomorphism $F$ of the ambient space, regardless the embedding, the fact that we can pick $F$ to be a contracting automorphism is not trivial and may depend on the embedding chosen. In theorem \ref{thm:GoodEmb}, we completely analyse this problem and show that any extension $F$ is automatically a contracting automorphism if we pick an embedding of $(X, 0)$ realizing the embedding dimension.
	To do so, the main idea in a reasoning by contradiction is to use Cauchy estimates on arcs of the singularity not lying in the stable manifold of the extension $F$.
	
	The first step being achieved, we can now consider a contracting automorphism $F$ on $(\CC^d, 0)$ in Poincaré-Dulac normal form, and study its invariant subvarieties.
	Two main difficulties arise: the presence of resonances (which do not appear in \cite{OrlikWagreich}), and having to work with subvarieties of arbitrary codimension.
	The techniques used here are mainly combinatorial, to show that the ideal defining the invariant subvariety is generated by weighted homogeneous polynomials.
	
	Some of the methods developed here appear independently in the recent work of Bivià-Ausina, Kourliouros and Ruas  \cite{BAKR}, where the authors study the relation between the quasi-homogeneity of a pair $(f, X)$ and the finiteness of its Milnor number. In op. cit., the authors use a famous result of Saito \cite{Saito1971} proving the equivalence between quasi-homogeneity of a singularity and the equality of its Milnor and Tjurina numbers. 
	
	It would be interesting to understand more deeply how these results can be related.

	The paper is organized as follows: in section 2 we introduce the main objects that will be studied and show some useful properties. In section 3 we prove that for any singularity $(X, 0)$, there exists an embedding $\iota : (X, 0) \to (\mathbb{C}^d, 0)$ such that any extension of a contracting automorphism of $(X, 0)$ is a contracting automorphism of $(\mathbb{C}^d, 0)$. In section 4, we show that ideals of holomorphic functions that are invariant under the action of a contracting automorphism are generated by quasi-homogeneous polynomials in some specific coordinates of $(\CC^d, 0)$.

	\subsubsection*{Acknowledgements}
	The author is deeply grateful to André Belotto da Silva and Matteo Ruggiero for their numerous advices during the realization of this project. Many thanks also go to Marie-Claude Arnaud
	for very interesting discussions on the different types of convergence.
	The author is partially supported by the project ``Plan d’investissements France 2030", IDEX UP ANR-18-IDEX-0001.

	\section{Preliminaries}\label{sect:Prel}
	
	In this section, we introduce the main tools that will be used throughout this paper.
	
	\subsection{Contracting automorphisms}\label{ssect:Contr}
	
	As could be seen in the introduction, contracting automorphisms play an important role in this paper.
	
	\begin{definition}[Contracting automorphism]\label{def:contr}
		Let $(X, 0)$ be a complex singularity, and $F: (X, 0) \to (X, 0)$ be an automorphism. We say that $F$ is a \emph{contracting automorphism} if for all open neighbourhood $W$ of $0$ in $X$, there exists an open neighbourhood $U$ of $0$ in $W$ such that:
		\begin{itemize}
			\item $F (U) \Subset U$,
			\item $\bigcap_{n \geq 0} F^{n} (U) = {0}$.
		\end{itemize}
	\end{definition}
	
	As we intend to generalize the result of Favre and Ruggiero \cite{FavreRuggiero}, we state the definition that they used. However we will need a stronger characterisation of contracting automorphisms later:
	
	\begin{definition}[Uniformly contracting automorphism]\label{def:unifcontrauto}
		Let $(X, 0)$ be a complex singularity, and $F: (X, 0) \to (X, 0)$ be an automorphism.
		
		It is \emph{uniformly contracting} if it is contracting and for all $V$ neighbourhood of $0$, there exists $n \in \mathbb{N}$ such that $F^{n} (U) \subset V$.
	\end{definition}
	
	These two definitions of contraction are, in fact, equivalent:
	
	\begin{prop}\label{prop:CUC}
		Let $(X, 0)$ be a singularity and $F : (X, 0) \to (X, 0)$ be a holomorphic diffeomorphism, then $F$ is uniformly contracting if and only if $F$ is contracting.
	\end{prop}
	
	\begin{proof}
		It is easy to prove that an uniformly contracting map is also contracting, so we only need to consider the converse. We will prove it by contradiction.
		
		Let us assume that $F$ is contracting and there exists an open neighbourhood $V$ of $0$ such that for all $n \in \NN$, we have $F^n (U) \nsubseteq V$.
		
		Let $K = \overline{F (U)}$, it is a compact set that satisfies $F (K) \subset K$. We also have for all $n \in \NN$, $F^n (K) \nsubseteq V$, so for all $n$ there exists $x_n \in F^n (K) \setminus V$. The sequence $(x_n)_n$ converges to a point $x \notin V$ which is in $F^n (K)$ for all $n$, so it belongs to the intersection  $\bigcap_{n \geq 0} F^{n} (K) = \bigcap_{n \geq 0} F^{n} (U)$.
	\end{proof}
	
	In the case of a holomorphic diffeomorphism $F : (\mathbb{C}^d, 0) \to (\mathbb{C}^d, 0)$, the contraction condition that we stated above is also equivalent to the eigenvalues of $D_0 F$ being of modulus smaller than $1$ (cf. contraction principle (\cite{MNTU}, p.219)). In section \ref{sect:Ext}, we obtain a similar result for automorphisms on singularities.
	
	\subsection{Embeddings of a singularity}\label{ssect:Emb}
	
	Recall that a singularity germ $(X, 0)$ admits a local embedding $\iota : (X, 0) \hookrightarrow (\mathbb{C}^d, 0)$ for $d$ sufficiently big, that is $\iota$ is a closed immersion. In particular the local ring $\mathcal{O}_{X, 0}$ is isomorphic to $\mathcal{O}_{\mathbb{C}^d, 0} / \Ideal_{\iota}$ for some ideal $\Ideal_{\iota} \subset \mathfrak{m}_{\CC^d, 0} \subset \mathcal{O}_{\mathbb{C}^d, 0}$ where $\mathfrak{m}_{\CC^d, 0}$ is the maximal ideal of $\mathcal{O}_{\CC^d, 0}$. Note that the ideal $\Ideal_\iota$ depends on the embedding.
	
	\begin{definition}[Embedding dimension]\label{def:embdim}
		The \emph{embedding dimension} of $(X, 0)$ is the minimal $d \in \mathbb{N}$ such that there exists an embedding $\iota :(X,0) \to (\mathbb{C}^d, 0)$. 
	\end{definition}
	
	\begin{prop}\label{prop:m2}
		Let $(X, 0)$ be a complex singularity, and $\iota : (X, 0) \to (\mathbb{C}^{d}, 0)$ an embedding of $(X, 0)$ in $(\mathbb{C}^d, 0)$. Then $d$ is the embedding dimension of $(X, 0)$ if and only if $\Ideal_{\iota} \subset \mathfrak{m}_{\CC^d, 0}^2$, where $\mathfrak{m}_{\CC^d, 0}$ is the ideal of holomorphic functions of $(\CC^d, 0)$ vanishing at $0$.
	\end{prop}
	
	This result is already known (see \cite{Gunning90}[Chapter 1, Corollary 16]) but we give a proof here for the sake of completeness.
	
	\begin{proof}
		\underline{$\Longrightarrow$:}
		
		We prove it by contraposition: let us suppose that $\Ideal_\iota \nsubseteq \mathfrak{m}_{\CC^d, 0}^2$. Then there exists a holomorphic function $g : (\mathbb{C}^d, 0) \to (\mathbb{C}, 0)$ vanishing on $X$ and with non trivial linear part. We can reorder the coordinates so that $\frac{\partial g}{\partial x_d} (0) \neq 0$. Then by the implicit function theorem, there exists a neighbourhood $V$ of the origin and a holomorphic function $\phi : \mathbb{C}^{d-1} \to \mathbb{C}$ such that for all $x \in V$, $g (x) = 0$ if and only if $x_d = \phi (x_1, \ldots, x_{d-1})$.
		
		Let $\Phi : \mathbb{C}^{d-1} \to \mathbb{C}^d$ be defined by $\Phi (y) = (y, \phi (y))$ and $\Pi : \mathbb{C}^d \to \mathbb{C}^{d-1}$ be the restriction on the first $d-1$ coordinates. Then $\Pi$ induces a biholomorphism between $X$ and $\Pi (X)$, and its inverse is $\Phi_{|\Pi (X)}$. We thus have an embedding of $X$ in $\mathbb{C}^{d-1}$, which implies that $d$ is not minimal among the natural numbers $n$ such that there exists an embedding of $(X, 0)$ in $(\mathbb{C}^n, 0)$.
		
		\underline{$\Longleftarrow$:}
		
		Let $\tilde{\iota} : (X, 0) \to (\mathbb{C}^n, 0)$ be another embedding of $(X, 0)$. We want to prove that $n \geq d$.
		
		The map $(\tilde{\iota} \circ \iota^{-1})_{|\iota (X)} : (\iota (X), 0) \to (\mathbb{C}^n, 0)$ is holomorphic, so there exists a holomorphic function $\phi : (\mathbb{C}^d, 0) \to (\mathbb{C}^n, 0)$ such that $\phi_{|\iota (X)} = (\tilde{\iota} \circ \iota^{-1})_{|\iota (X)}$.
		
		There exists also a holomorphic function $\psi : (\mathbb{C}^n, 0) \to (\mathbb{C}^d, 0)$ such that $\psi_{|\tilde{\iota} (X)} = (\iota \circ \tilde{\iota}^{-1})_{|\tilde{\iota} (X)}$, and they satisfy:
		$$
		\psi \circ \phi_{|\iota (X)} = id_{|\iota (X)} \text{ and } \phi \circ \psi_{\tilde{\iota} (X)} = id_{\tilde{\iota} (X)}.
		$$
		
		Let $\Pi_i : \mathbb{C}^d \to \mathbb{C}$ be the projection on the $i$-th coordinate. Then $\Pi_i \circ \psi \circ \phi_{|\iota (X)} = x_i$ so there exists $f_i \in \Ideal_{\iota} \subset \mathfrak{m}_{\CC^d, 0}^2$ such that $\Pi_i \circ \psi \circ \phi = x_i + f_i$. It implies that $D_0 (\psi \circ \phi) = I_d$ since all $f_i$ are in $\mathfrak{m}_{\CC^d, 0}^2$, so $\psi \circ \phi$ is locally invertible. But then $\psi$ has to be locally surjective, which gives $n \geq d$.
	\end{proof}

	\subsection{Quasi-homogeneous singularities}\label{ssect:Qhs}
	
	We now introduce the second object that appears in the main theorem, and we show its link with contracting automorphisms.
	
	\begin{definition}[Weighted homogeneous polynomial]\label{def:whp}
		Let $(n_1, \ldots, n_d)$ be non-zero integers. A polynomial $P \in \mathbb{C}[X_1, \ldots, X_d]$ is a \emph{weighted homogeneous polynomial} for the weights $(n_1, \ldots, n_d)$ if there exists $n \in \mathbb{N}$ such that:
		$$
		P (t^{n_1} x_1, \ldots, t^{n_d} x_d) = t^n P (x_1, \ldots, x_d)
		$$
		for all $t \in \mathbb{C}$ and $(x_1, \ldots, x_d) \in \mathbb{C}^d$.
	\end{definition}
	
	\begin{definition}[Quasi-homogeneous singularity]\label{def:qhs}
		A singularity $(X, 0)$ is said to be \emph{quasi-homogeneous} if there exists an embedding $\iota : (X, 0) \hookrightarrow (\CC^d, 0)$ and weights $(n_1, \ldots, n_d)$ such that $\Ideal_{\iota}$ is generated by weighted homogeneous polynomials of weights $(n_1, \ldots, n_d)$. 
	\end{definition}
	
	Such a singularity carries contracting automorphisms: for any $t$ of modulus smaller than $1$, the map $f : (x_1, \ldots, x_d) \mapsto (t^{n_1} x_1, \ldots, t^{n_d} x_d)$ is contracting and the singularity is invariant under its action.

	\section{Extension of automorphisms}\label{sect:Ext}
	
	In this section, we prove that any singularity admits an embedding such that contracting automorphisms on the singularity extend as contracting automorphisms of the ambient space.

	\begin{definition}[Extension of function]\label{def:ext}
		Given an embedding $\iota :(X,0) \to (\mathbb{C}^d, 0)$ and $f \in \mathcal{O}_{X,0}$, an \emph{extension} of $f$ to $(\mathbb{C}^d, 0)$ is a function $F \in \mathcal{O}_{\mathbb{C}^d, 0}$ such that $f = F \circ \iota$.
		
		Let $f : (X, 0) \to (X,0)$ be an automorphism. An extension of $f$ to $(\mathbb{C}^d, 0)$ is a map $F : (\mathbb{C}^d, 0) \to (\mathbb{C}^d, 0)$ such that $F \circ \iota = \iota \circ f$.
	\end{definition}
	
	\begin{definition}[Good embedding]\label{def:goodemb}
		Let $(X, 0)$ be a complex singularity, and $\iota : (X, 0) \to (\mathbb{C}^{d}, 0)$ an embedding of $(X, 0)$ into $(\mathbb{C}^d, 0)$. We say that $\iota$ is a \emph{good embedding} if for any automorphism $f$ of $(X,0)$, all of its extensions are automorphisms of $(\mathbb{C}^d, 0)$.  We say that $\iota$ is a \emph{good contracting embedding} if it is a good embedding and for any contracting automorphism $f$ of $(X,0)$, all of its extensions are contracting automorphisms of $(\mathbb{C}^d, 0)$.
	\end{definition}
	
	\begin{thm}\label{thm:GoodEmb}
		Let $(X, 0)$ be a complex singularity, and $\iota : (X, 0) \to (\mathbb{C}^{d}, 0)$ an embedding of $(X, 0)$ in $(\mathbb{C}^d, 0)$. Then the following are equivalent:
		\begin{enumerate}
			\item $d$ is the embedding dimension of $(X, 0)$,
			\item $\iota$ is a good embedding,
			\item $\iota$ is a good contracting embedding.
		\end{enumerate}
	\end{thm}
	
	The equivalence between the first two properties was already known (see \cite{Gunning90}[Chapter 1, Theorem 10]) but we give another proof in the next subsection for the sake of completeness. We will prove then in subsection \ref{ssect:GoodCE} that being a good embedding and being a good contracting embedding are equivalent.
	
	\subsection{Good embedding}\label{ssect:GoodE}
	
	In this section, we prove the equivalence of the first two properties.

	\underline{$1 \implies 2$}
	
	Let $\iota : (X, 0) \to (\CC^d, 0)$ be an embedding such that $d$ is the embedding dimension. Then we also have $\Ideal_\iota \subset \mathfrak{m}^2$ (cf. Proposition \ref{prop:m2}).
	
	Let $f : (X, 0) \to (X, 0)$ be an automorphism and $g : (X, 0) \to (X, 0)$ be its inverse. We know that $f \circ g = g \circ f = id_X$.
	
	Let $F$ and $G$ be two extensions of $f$ and $g$ to $(\mathbb{C}^d, 0)$.
	
	Let $\Pi_i : \mathbb{C}^d \to \mathbb{C}$ be the projection on the $i$-th coordinate. Then the function $\Pi_i \circ G \circ F$ is holomorphic and is equal to $x_i$ when restricted to $X$. Then there exists $\phi_i \in \Ideal_\iota \subset \mathfrak{m}_{\CC^d, 0}^2$ such that $\Pi_i \circ G \circ F = x_i + \phi_i$. This shows that $G \circ F$ is equal to the identity up to applications in $\mathfrak{m}^2$ in each coordinate. It is thus a local biholomorphism as its linear part is invertible, so there exists $H_1 : (\mathbb{C}^d, 0) \to (\mathbb{C}^d, 0)$ such that $H_1 \circ G \circ F = id$. This shows that $F$ is left invertible.
	
	With a similar argument, we can find $H_2 : (\mathbb{C}^d, 0) \to (\mathbb{C}^d, 0)$ such that $F \circ G \circ H_2 = id$, so $F$ is also right invertible and is thus a biholomorphism of $(\mathbb{C}^d, 0)$.
	
	\underline{$2 \implies 1$}
	
	We prove it by contraposition. Let us assume that $d$ is not the embedding dimension, that is $\Ideal_{\iota} \not\subset \mathfrak{m}_{\CC^d, 0}^2$ (Proposition \ref{prop:m2}). Then there exists a function $g \in \mathcal{O}_{\CC^d, 0}$ vanishing on $(X, 0)$ and with non trivial linear part $G \in M_{1,d} (\CC)$.
	
	Let $f$ be an automorphism of $(X, 0)$ and $F$ be an extension to $(\CC^d, 0)$. Then for any constants $l_1, \ldots, l_d$, the map:
	$$
	(\Pi_1 \circ F + l_1 g, \ldots, \Pi_d \circ F + l_d g)
	$$
	is also an extension of $f$ to $(\CC^d, 0)$. Moreover, its linear part is $D_0 F + L G$ where $L = \begin{pmatrix}
		l_1\\
		\ldots\\
		l_d
	\end{pmatrix}$. Then the determinant of the linear part is:
	
	\begin{align*}
		\det (D_0 F + LG) &= \det(D_0 F_1 + l_1 G, \ldots, D_0 F_d + l_d G)\\
		&= \det (D_0 F) + \sum_{1 \leq i \leq d} l_i \det (D_0 F_1, \ldots, D_0 F_{i-1}, G, D_0 F_{i+1}, \ldots, D_0 F_d).
	\end{align*}
	
	Notice that $D_0 F$ is invertible, \textit{i.e.} $(D_0 F_j)_j$ is a basis. As $G \neq 0$, there exists $i$ such that $\det (D_0 F_1, \ldots, D_0 F_{i-1}, G, D_0 F_{i+1}, \ldots, D_0 F_d) \neq 0$. We deduce the existence of $L$ such that $\det (D_0 F + L G) = 0$, which is a contradiction.

	\subsection{Good contracting embedding}\label{ssect:GoodCE}

	If $\iota_1$ and $\iota_2$ are two good embeddings of $(X, 0)$, then for any automorphism $f$ of $(X, 0)$, the linear parts $D_0 F_1$ and $D_0 F_2$ of two extensions of $f$ are actually conjugated.
	Since being contracting on $(\CC^d, 0)$ only depends on the eigenvalues of $D_0 F$, we only need to prove that one good embedding is a good contracting embedding. We can thus change coordinates in $(\mathbb{C}^d, 0)$, and the good contracting embedding condition will still be satisfied.
	
	We can now finish the proof of the theorem.
	
	\begin{proof}[Proof of theorem \ref{thm:GoodEmb}]
		Let $\iota : (X, 0) \to (\mathbb{C}^d, 0)$ be a good embedding of $(X, 0)$ and let $f : (X, 0) \to (X, 0)$ be a contracting automorphism of $(X, 0)$. Let $F$ be an extension of $f$ to $(\mathbb{C}^d,0)$.
		
		Let $(\lambda_1, \ldots, \lambda_d)$ be the eigenvalues of $D_0 F$. Then $F$ is contracting on $(\mathbb{C}^d, 0)$ if and only if for all $1 \leq i \leq d$, we have $|\lambda_i| < 1$.
		
		Let us assume that there exists $k < d$ such that:
		\begin{itemize}
			\item $|\lambda_i|<1$ for all $i \leq k$,
			\item $|\lambda_i| \geq 1$ for all $i > k$.
		\end{itemize}
		
		By the stable manifold theorem \cite{Abate2001}, we know that there exists a holomorphic stable manifold invariant under the action of $F$ and tangent to the subspace generated by the eigenvectors of $D_0 F$ for which the eigenvalue is of modulus smaller than $1$.
		
		We can thus make a holomorphic change of coordinates so that the stable manifold is given by $\{x_{k+1} = \ldots = x_d = 0\}$. In such coordinates, we have: 
		\begin{equation}\label{eq:Fi}
			F_i (x_1, \ldots, x_k, 0, \ldots, 0) = 0 \quad \forall i>k, (x_1, \ldots, x_k) \in \mathbb{C}^k.
		\end{equation}
		
		We can moreover make a linear change of coordinates so that $D_0 F$ is in lower Jordan normal form (it is lower triangular) and \eqref{eq:Fi} is still satisfied.
		
		As $\iota$ is a good embedding, we know that we cannot embed $(X, 0)$ in the stable variety because it would contradict the minimality of $d$ as the smallest natural number such that we can embed $(X,0)$ in $(\mathbb{C}^n, 0)$. So there exists a point $p = (p_1, \ldots, p_d) \in \iota (X, 0)$ such that there exists $i > k$ for which $p_i \neq 0$.
		
		Then we can find a holomorphic curve $\phi : (\mathbb{C}, 0) \to (\mathbb{C}^d, 0)$ such that $\phi (0) = 0$ and $p \in Im (\phi)$. For all $i$, we can write:
		\begin{equation}\label{eq:phi}
			\phi_i (t) = c_i t^{n_i} + o (t^{n_i})
		\end{equation}
		with $c_i \neq 0$.
		
		Let $n = \min \{n_i, i>k\}$ and $l = \min \{i > k, n_i = n\}$.
		
		Then let us prove the following lemma:
		
		\begin{lem}
			Let $m \in \mathbb{N}$ be an integer.
			\begin{itemize}
				\item For all $k < i < l$, we have: $(F^m \circ \phi)_i (t) = o (t^n).$
				\item For $i = l$, we have: $(F^m \circ \phi)_l (t) = c_l \lambda_l^{m} t^{n} + o (t^n).$
				\item For all $l < i <d$, we have: $(F^m \circ \phi)_i (t) = O (t^n).$
			\end{itemize}
		\end{lem}
		
		\begin{proof}
			We prove it by induction on $m$.
			
			For $m = 0$, these properties are true by \eqref{eq:phi} and the definitions of $n$ and $l$.
			
			Let us suppose that they are true for the rank $m$.
			
			Then for all $i$, we have 
			
			\begin{align*}
				(F^{m+1} \circ \phi)_i (t) &= F_i ((F^m \circ \phi)_1 (t), \ldots, (F^m \circ \phi)_d (t))\\
				&= \lambda_i (F^m \circ \phi)_i (t) + \epsilon_i (F^m \circ \phi)_{i-1} (t) + \text{ h.o.t.}
			\end{align*}
			where h.o.t consists of terms of higher order in $t$.
			
			We first remark that for all $i > k$, equation \eqref{eq:Fi} implies that $F_i$ does not contain any monomial of the form $x_1^{a_1} \ldots x_k^{a_k}$. Then every monomial in h.o.t. is of the form $x_j P$ with $j > k$ and $P (0) = 0$. By the induction hypothesis, we know that $x_j = O (t^n)$, so the higher order terms are actually a $o (t^n)$.
			
			Finally, we can see that $\epsilon_{k+1} = 0$ as $|\lambda_k| < 1 \leq |\lambda_{k+1}|$.
			
			We can thus conclude:
			
			\begin{itemize}
				\item For all $k < i < l$, $(F^{m+1} \circ \phi)_i (t) = \lambda_i o (t^n) + \epsilon_i o(t^n) + o (t^n) = o (t^n).$
				\item For $i = l$, we have: $(F^{m+1} \circ \phi)_l (t) = \lambda_l c_l \lambda_l^{m} t^{n} + \epsilon_l o (t^n) + o (t^n) = c_l \lambda_l^{m+1} t^{n} + o (t^n) .$
				\item For all $l < i <d$, $(F^{m+1} \circ \phi)_i (t) = \lambda_i O (t^n) + \epsilon_i O(t^n) + o (t^n) = O (t^n).$
			\end{itemize}
		\end{proof}
		
		Let $C = \phi (\mathbb{S}_1)$. As it is compact and $F$ is contracting on $X$, we know that for all $\epsilon > 0$ there exists $m \in \mathbb{N}$ such that $F^m (C) \subset B (0, \epsilon)$. In particular, it implies:
		
		\begin{align*}
			\lvert c_l \lambda_l^{m} \rvert &= \left| \frac{1}{2i\pi} \int_{\mathbb{S}_1} \frac{(F^m \circ \phi)_l (t)}{t^{n+1}} dt \right|
			\leq \frac{1}{2i\pi} \int_{\mathbb{S}_1} \left| \frac{(F^m \circ \phi)_l (t)}{t^{n+1}} \right| dt
			\leq \frac{\epsilon}{2\pi}.
		\end{align*}
		
		But then $|\lambda_l| < 1$, which is a contradiction.
	\end{proof}

	\section{Invariant ideals}\label{sect:InvId}
	
	The goal of this part is to prove the following theorem:
	
	\begin{thm}\label{thm:InvId}
		Let $F$ be a contracting diffeomorphism of $(\mathbb{C}^d, 0)$ and $\Ideal$ be an ideal of holomorphic functions which is invariant under the action of $F$. Then there exists holomorphic coordinates of $(\mathbb{C}^d, 0)$ in which $\Ideal$ is generated by weighted homogeneous polynomials $P^{(1)}, \ldots, P^{(r)}$ such that, for all $i \in \{1, \ldots, r\}$, the ideal $\langle P^{(1)}, \ldots, P^{(i)} \rangle$ is invariant under the action of $F$.
	\end{thm}
	
	We first choose holomorphic coordinates in $(\CC^{n}, 0)$ to get a normal form for $F$. Then we build an order and an equivalence relation on $\NN^d$, which enable us to decompose any map $(\CC^{n}, 0) \to (\CC, 0)$ into a sum of weighted homogeneous polynomials. Let $\Ideal$ be an ideal invariant under the action of $F$, we pick a minimal number of generators $\phi^{(1)}, \ldots, \phi^{(r)}$. After a linear change of the generators, we show that we can extract a weighted homogeneous polynomial $P^{(i)}$ (which is not the initial part) from each $\phi^{(i)}$ such that $\Ideal$ is generated by these polynomials. We prove this by double inclusion, first thanks to the definition of $P^{(i)}$ and then using the minimality of the number of generators.
	
	\subsection{Normal form for $F$}\label{ssect:NF}
	
	In this section, we change the coordinates of $(\mathbb{C}^d, 0)$ in order to get a more convenient form for $F$.
	
	\begin{definition}[Nicely ordered spectrum]\label{def:lambda}
		Let $\lambda$ be the spectrum of $D_0 F$. We say that it is \emph{nicely ordered} if the eigenvalues are ordered by their modulus : $1 > \abs{\lambda_i} \geq \abs{\lambda_j} > 0$ for $i \leq j$ (all the eigenvalues are of modulus smaller than $1$ as $F$ is contracting).
	\end{definition}
	
	After performing a linear change of coordinates, we can assume that $D_0 F$ is in lower Jordan normal form and that its diagonal $\lambda$ is nicely ordered, so we can write:
	$$
	D_0 F (x_i) = \lambda_i x_i + \epsilon_i x_{i-1}
	$$
	with $\epsilon_i \in \{0,1\}$ can be non zero only if $\lambda_i = \lambda_{i-1}$.

	\begin{definition}[Resonant monomial]\label{def:res}
		A monomial $x^{\Alp}$ is said to be \emph{$\lambda$-resonant} with respect to the $i$-th coordinate if it satisfies $|\Alp| \geq 2$ or $\Alp = \mathds{1}_{i-1} = (0, \ldots, 0, 1, 0, \ldots, 0)$ and:
		$$
		\lambda^{\Alp} = \lambda_i
		$$	
		where $\lambda^{\Alp} = \lambda_1^{\Alp_1} \ldots \lambda_d^{\Alp_d}$.
		
		We denote by $R_i = \{\Alp \in \NN^d, \lambda^{\Alp} = \lambda_i\}$ the set of exponents that are $\lambda$-resonant with respect to the $i$-th coordinate.
		
	\end{definition}
	
	\begin{definition}[Poincaré-Dulac normal form]\label{def:pdnf}
		A germ of diffeomorphism $F : (\mathbb{C}^d, 0) \to (\mathbb{C}^d,0)$ is in Poincaré-Dulac normal form if $D_0 F$ is in lower Jordan normal form, its diagonal is nicely ordered and $F - D_0 F$ admits only resonant monomials. We will call the diagonal $\lambda$ of $D_0 F$ the \emph{ordered spectrum} of $F$.
	\end{definition}
	
	In order to get a normal form for the automorphism $F$, we now recall the classic Poincaré-Dulac theorem (\cite{Dulac1912}, \cite{Berteloot06}[Chapter 4], \cite{RosayRudin}, \cite{Stern57}).
	
	\begin{thm}[Poincaré-Dulac]\label{thm:PD}
		Let $F$ be a contracting automorphism of $(\mathbb{C}^{d}, 0)$. Then $F$ is holomorphically conjugated to an automorphism in Poincaré-Dulac normal form.
	\end{thm}
	
	\begin{rem}\label{rem:PD}
		Our choice of a lower triangular Jordan normal form for $D_0 F$ and of the order on the eigenvalues by non-increasing modulus implies that any germ $F$ in Poincaré-Dulac normal form is triangular, i.e., of the form:
		$$
		F (x_1, \ldots, x_d) = (f_1 (x_1), f_2 (x_1, x_2), \ldots, f_d (x_1, \ldots, x_d)).
		$$
		In fact, any monomial $x^{\Alp}$ with $\Alp \in R_i$ satisfies $\Alp_k = 0$ for all $k \geq i$. 
		
		Hence, for all $i$, we can write: 
		\begin{equation}\label{eq:fi}
			f_{i} (x_1, \ldots, x_i) = \lambda_i x_i + \sum_{\Alp \in R_i} \nu_{i,\Alp} x^{\Alp},
		\end{equation}
		for some $\nu_{i,\Alp} \in \CC$.
	\end{rem}
	
	From now on, we will assume that $F$ is in Poincaré-Dulac normal form.

	\subsection{An order on $\mathbb{N}^d$}\label{ssect:order}

	We see in the previous section that the definition of resonant monomials relies on the coefficients $\lambda^{\Alp}$, where $\lambda$ is an ordered spectrum of $F$ and $\Alp \in\NN^d$. Similarly, it will make sense to use these numbers to order the elements of $\mathbb{N}^d$ in order to decompose any holomorphic function in a sum of weighted homogeneous polynomials. Some of these polynomials will be used later to generate the ideal $\Ideal$.
	
	\begin{definition}[$\lambda$-order]\label{def:ord}
		Let $F$ be a contracting automorphism and $\lambda$ be its ordered spectrum.
		We define the \emph{$\lambda$-order} over $\mathbb{N}^d$ as follows: for any $\Alp$, $\Bet \in \mathbb{N}^d$, we say that $\Alp \succ_{\lambda} \Bet$ if either $|\lambda^{-\Alp}| > |\lambda^{-\Bet}|$ or $|\lambda^{\Alp}| = |\lambda^{\Bet}|$ and $\Alp > \Bet$ in the lexicographic order (\textit{i.e.} $\Alp_1>\Bet_1$, or $\Alp_1 = \Bet_1$ and $\Alp_2 > \Bet_2$, etc).
		
		In other terms, the $\lambda$-order is the lexicographic order on $ (|\lambda^{-\Alp}|, \Alp_1, \ldots, \Alp_n)$.
	\end{definition}
	
	We will omit the index on $\succ_{\lambda}$ when the choice of the ordered spectrum is clear from context.
	
	The $\lambda$-order is a total order on $\mathbb{N}^d$, which is also compatible with the sum: for all $\Alp \succeq \Bet$ and $\Alp' \succeq \Bet'$, we have: $\Alp + \Alp' \succeq \Bet + \Bet'$.
	
	\begin{definition}[$\lambda$-weight]\label{def:eqr}
		Let $\lambda$ be the ordered spectrum of $F$. We say that two n-tuples $\Alp, \Alp' \in \NN^d$ have same \emph{$\lambda$-weight} if
		$$
		\lambda^{\Alp - \Alp'} = 1.
		$$
		
		This actually defines an equivalence relation $\sim$ on $\NN^d$ and we denote by $\Gamma = \NN^d / \sim$ the quotient set. We can order the equivalent classes of $\Gamma$ according to their smallest element: we say that $\gamma \prec \gamma'$ if $\min \{\Alp, [\Alp] = \gamma\} \prec \min \{\Alp, [\Alp] = \gamma'\}$, where the minimum is taken with respect to the $\lambda$-order.
	\end{definition}
	
	From now on, we will sometimes abuse notation and denote by $\lambda^\gamma$ the value $\lambda^{\Alp}$ for any representative $\Alp$ of the class $\gamma$.
	
	\begin{definition}[$\log_{\lambda}$]\label{def:log}
		The map $\gamma \mapsto \lambda^{\gamma}$ defines a bijection between $\Gamma$ and its image $\lambda^{\Gamma}$. We denote by $\log_{\lambda}$ its inverse. 
	\end{definition}
	
	\begin{rems}\label{rems:ec}\hfill
		
		\begin{enumerate}
			\item As all the eigenvalues in $\lambda$ are of modulus smaller than $1$, the equivalent classes in $\NN^d$ are all finite.
			\item The set $(\Gamma, \succeq)$ is order isomorphic to $(\mathbb{N}, \geq)$. We can thus do inductions on $\Gamma$.
			\item The sum on $\mathbb{N}^d$ induces a sum on $\Gamma$ which satisfies $\gamma + \gamma' \succeq \gamma$ for all $\gamma, \gamma'$.
		\end{enumerate}
	\end{rems}
	
	\begin{ex}
		The equivalence classes in $\NN^d$ actually depend on the choice of the order of the eigenvalues in the ordered spectrum $\lambda$: pick $\lambda = (-1/2, i/2, i/2)$ and $\mu = (i/2, i/2, -1/2)$, the $(0,1,0) \sim_\lambda (0,0,1)$ but they are not of the same $\mu$-weight.
	\end{ex}
	
	\subsection{$\lambda$-gradation on $\CC[[x]]$}
	
	\begin{definition}[$\lambda$-homogeneous polynomials]\label{def:Hog}
		Let $\gamma \in \Gamma$, we will note $\Ho_\gamma$ the vector space of polynomials generated by the monomials $x^{\Alp}$ such that $[\Alp] = \gamma$. It is finite dimensional as the equivalent classes in $\NN^d$ are finite.
		
		We say that a polynomial $P$ is \emph{$\lambda$-homogeneous} if $P \in \Ho_\gamma$ for some $\gamma$, called the $\lambda$-degree of $P$.
	\end{definition}
	
	These spaces define a gradation of the vector space $\CC[[x]]$ which behaves well with the structure of $\CC$-algebra on $\CC[[x]]$:
	
	\begin{lem}\label{lem:prod}
		For any $\gamma, \gamma' \in \Gamma$, we have $\Ho_\gamma \cdot \Ho_{\gamma'} \subseteq \Ho_{\gamma+\gamma'}$.
	\end{lem}
	
	\begin{proof}
		This simply comes from $\lambda^{\gamma} \lambda^{\gamma'} = \lambda^{\gamma + \gamma'}$.
	\end{proof}
	
	\begin{rem}\label{rem:ex}
		There is however not equality in the inclusion given in the lemma above: if $\lambda = (1/2, 1/4)$, then $\Ho_{[(1, 0)]}$ is generated by $x_1$ but $\Ho_{[(2, 0)]}$ is generated by $x_1^2$ and $x_2$.
	\end{rem}
	
	The monomials appearing once we develop $x^{\Bet} \circ F$ are in the equivalent class of $\Bet$ and are larger or equal than $\Bet$ (for the order $\succ_{\lambda}$). More precisely, we get:
	
	\begin{lem}\label{lem:Bet}
		Let $F$ be a contracting automorphism on $(\CC^d, 0)$ in Poincaré-Dulac normal form, and denote by $\lambda$ its ordered spectrum.
		For any $\Bet \in \mathbb{N}^d$, we have:
		$$
		x^{\Bet} \circ F = 
		\lambda^{\Bet} x^{\Bet} + \sum_{\substack{[\Alp] = [\Bet] \\ \Alp \succ \Bet }} \nu_{\Alp} x^{\Alp} 
		$$
		for some $\nu_{\Alp} \in \CC$ (which depend on $\Bet$).
	\end{lem}
	
	\begin{proof}
		We first recall:
		\begin{equation}\tag{\ref{eq:fi}}
			x_i \circ F = \lambda_i x_i + \sum_{\Alp \in R_i} \nu_{i,\Alp} x^{\Alp},
		\end{equation}
		for some $\nu_{i,\Alp} \in \CC$.
		
		One can remark that $R_i$ has been defined precisely as the equivalence class of $\mathds{1}_i = (0, \ldots, 0, 1, 0, \ldots, 0)$ with the $1$ at the $i$-th position. Moreover, $\mathds{1}_i$ is minimal for the order $\succ_{\lambda}$ in its equivalence class. We indeed know (cf. remark \ref{rem:PD}) that for all $\Alp \in R_i$, there exists $k<i$ such that $\alpha_k \neq 0$, which implies that $\Alp$ is bigger than $\mathds{1}_i$ for the lexicographic order.
		
		We thus showed that all $x_i \circ F \in \Ho_{[\mathds{1}_i]}$.
		
		Now take $\Bet = (\beta_1, \ldots, \beta_d) \in \NN^d$, then:
		$$
		x^{\Bet} \circ F = \prod_{i=1}^{d} \left(\lambda_i x_i + \sum_{\Alp \in R_i} \nu_{i,\Alp} x^{\Alp} \right)^{\beta_i}.
		$$
		
		Then, by lemma \ref{lem:prod}, the right-hand term is in $\Ho_{[\sum \beta_i \mathds{1}_i]} = \Ho_{[\Bet]}$. Moreover, the monomials that appear once we develop all have exponents bigger than $\Bet$ because of the compatibility of $\succ_{\lambda}$ with the sum.
	\end{proof}
	
	The spaces $\Ho_\gamma$ are moreover stable under the action of $F$:
	
	\begin{lem}\label{lem:stab}
		For any $\gamma \in \Gamma$ we have $\Ho_\gamma \circ F = \Ho_\gamma$ where $\Ho_\gamma \circ F = \{\phi_\gamma \circ F, \phi_\gamma \in \Ho_\gamma\}$.
	\end{lem}
	
	\begin{proof}
		We first remark that the map $L : \phi_\gamma \mapsto \phi_\gamma \circ F$ is a linear map on $\Ho_\gamma$.
		
		Lemma \ref{lem:Bet} shows that the image of the basis vectors $x^{\Bet}$ for $[\Bet] = \gamma$ by $L$ is always in the vector space $\Ho_\gamma$, so $\Ho_\gamma \circ F \subseteq \Ho_\gamma$.
		
		Moreover if we order the monomials $x^{\Bet}$ which are a basis of $\Ho_\gamma$ according to our order $\prec_{\lambda}$ ($x^{\Alp} < x^{\Bet}$ if $\Alp \prec \Bet$), then the matrix representing $L$ is triangular with only $\lambda^\gamma$ on the diagonal, so $L$ is invertible and we get the equality $\Ho_\gamma \circ F = \Ho_\gamma$.
	\end{proof}

	Any holomorphic function $\phi \in \mathcal{O}_{\CC^d, 0}$ can be uniquely decomposed in its $\lambda$-homogeneous parts:
	$$
	\phi = \sum_{\gamma \in \Gamma} \phi_\gamma,
	$$
	with $\phi_\gamma \in \Ho_\gamma$. Moreover, we have:
	$$
	\phi \circ F = \sum_{\gamma \in \Gamma} \phi_\gamma \circ F,
	$$
	with $\phi_\gamma \circ F$ also in $\Ho_\gamma$ (cf. Lemma \ref{lem:stab}).
	
	We will note by $\phi_0$ the element of degree $[(0, \ldots, 0)]$, which is the constant term of $\phi$.
	
	Let $\phi = \sum_{\gamma \in \Gamma} \phi_\gamma$ and $\psi = \sum_{\gamma \in \Gamma} \psi_\gamma$ be elements in $\mathcal{O}_{\CC^d, 0}$ decomposed into the sum of their $\lambda$-homogeneous parts. Then:
	$$
	\phi \psi = \sum_{\delta \in \Gamma} \sum_{\substack{\gamma, \gamma' \in \Gamma \\ \gamma + \gamma' = \delta}} \phi_\gamma \psi_{\gamma'}
	$$
	with $\displaystyle \sum_{\gamma + \gamma' = \delta} \phi_\gamma \psi_{\gamma'} \in \Ho_{\delta}$ the $\lambda$-homogeneous part of degree $\delta$.
	
	\subsection{Invariance of $\Ideal$ under $\phi$}\label{ssect:Inv}
	
	Let $\Ideal$ be a finitely generated ideal of holomorphic functions which is invariant under the action of $F$, \textit{i.e.} for all $\phi \in \Ideal$, we have $\phi \circ F \in \Ideal$.
	
	Let $\phi^{(1)}, \ldots, \phi^{(r)}$ be a minimal set of generators of $\Ideal$ (\textit{i.e.} the number of generators is minimal). Since $\phi^{(i)} \circ F \in \Ideal$, there exists $A^i_j \in \mathcal{O}_{\CC^d, 0}$ for $i, j = 1, \ldots, r$ such that:
	\begin{equation}\label{eq:A}
		\phi^{(i)} \circ F = \sum_{j=1}^r A^i_j \phi^{(j)}.
	\end{equation}
	
	Let $A = (A^i_{j})_{i,j}$ and $A_0 = (A^i_{j,0})_{i,j} = (A^i_{j} (0))_{i,j}$ be the matrix with the constant terms. We can choose new generators for $\Ideal$ by making linear combinations of the $\phi^{(i)}$ so that the matrix $A_0$ is in lower Jordan normal form, thing that we will assume from now on.
	
	We recall that we can decompose:
	$$
	\phi^{(i)} = \sum_{\gamma \in \Gamma} \phi^{(i)}_\gamma
	$$
	with $\phi^{(i)}_\gamma \in \Ho_\gamma$.
	
	\begin{definition}\label{def:P(i)}
		For all $i \in \{1, \ldots, r\}$ we define
		\begin{equation}\label{eq:P}
			\gamma_i :=
			\begin{cases}
				\log_{\lambda} (A^i_{i, 0}) & \text{if } A^i_{i, 0} \in \lambda^{\Gamma}\\
				+ \infty & \text{otherwise}
			\end{cases}
			\qquad \text{ and } \qquad
			P^{(i)} :=
			\begin{cases}
				\phi^{(i)}_{\gamma_i} & \text{if } \gamma_i \in \Gamma \\
				0 & \text{if } \gamma_i = + \infty.
			\end{cases}
		\end{equation}
		
		Let $\Jdeal$ be the ideal generated by the $P^{(i)}$.
	\end{definition}
	
	\begin{rem}\label{rem:Pi}
		We will show that the set of polynomials $\{P^{(i)}\}_{1 \leq i \leq r}$ generates $\Ideal$. As the number of generators $r$ has been chosen to be minimal, this will imply that all $P^{(i)}$ are different from $0$.
	\end{rem}

	We can now reorder the coordinates so that if $i < j$, then $\gamma_i \leq \gamma_j$. This will not change the fact that $A_0$ is in lower Jordan normal form. This property will be useful later on as it will imply (see the proof of proposition \ref{prop:IJ}) that $P^{(i)} \circ F \in \langle P^{(j)} \rangle_{j \leq i}$.

	\begin{ex}
		Let $F : \CC^2 \to \CC^2$ be defined by: $F (x, y) = (\tfrac{x}{2}, \tfrac{y}{4})$.
		
		Set $\phi = x^2 - y$ and $\psi = x (x^2 - y) + x^5$, then the ideal $\langle \phi, \psi \rangle$ is invariant under the action of $F$ and we have:
		$$
		\phi \circ F = \tfrac{1}{4} \phi \: \text{ and } \: \psi \circ F = \tfrac{1}{32} \psi + (\tfrac{1}{4} - \tfrac{1}{32}) x \phi.
		$$
		
		The matrix $A_0$ is given by 
		$\begin{pmatrix}
			1/4 & 0\\
			0 & 1/32
		\end{pmatrix}$
		so the polynomials $P$ and $Q$ defined in definition \ref{def:P(i)} are $P = x^2 - y$ and $Q = x^5$. Notice that $Q$ is not the initial part of $\psi$, which is $x (x^2 - y)$.
	\end{ex}
	
	\begin{lem}\label{lem:Phi}
		Let $\phi^{(1)}, \ldots, \phi^{(r)}$ as above, and set $\Phi := \begin{pmatrix}
			\phi^{(1)}\\
			\vdots\\
			\phi^{(r)}
		\end{pmatrix}$. Then for any matrix $B \in  M_{r} (\mathcal{O}_{\CC^d, 0})$ such that $B \Phi = 0$, we have $B (0) = 0$.
	\end{lem}
	
	\begin{proof}
		If we pick $r$ holomorphic functions $u_1, \ldots, u_r$ such that $\sum_j u_j \phi^{(j)} = 0$, then $u_j (0) = 0$ for all $j$. Indeed, we could otherwise express on of the $\phi^{(i)}$ in terms of the others which would contradict the minimality of the number of generators.
	\end{proof}
	
	\begin{rem}\label{rem:uniqueP}
		While the polynomials $P^{(i)}$ depend on the choice of the generators $\phi^{(i)}$, they are uniquely determined once the generators are fixed. The lemma indeed shows that even if $A$ is not uniquely determined, its linear part $A_0$ is, and the polynomials $P^{(i)}$ only depend on $A_0$.
	\end{rem}
	
	\subsection{First inclusion}\label{ssect:IJ}
	
	\begin{prop}\label{prop:IJ}
		With the previous notations, we have $\Ideal \subseteq \Jdeal$.
	\end{prop}
	
	We will prove by double induction on $\gamma \in \Gamma$ and $i \in \{1, \ldots, r\}$ that all of the weighted homogeneous polynomials $\phi^{(i)}_\gamma$ are in $\Jdeal$. We can then conclude as $\phi^{(i)}$ is a linear combination of the $\phi^{(i)}_\gamma$. The proof relies on the following lemma:
	
	\begin{lem}\label{lem:kdeal}
		Let $\Kdeal$ be an ideal of holomorphic functions of $(\mathbb{C}^d, 0)$ such that $\Kdeal \circ F \subseteq \Kdeal$. Let $\gamma \in \Gamma$, $\phi \in \Ho_\gamma$ and $\zeta \neq \lambda^{\gamma}$ be a complex number.
		
		Then $\phi \in \Kdeal$ if and only if $\phi \circ F - \zeta \phi \in \Kdeal$.
	\end{lem}
	
	\begin{proof}
		$\Rightarrow$: Trivial.
		
		$\Leftarrow$: The function $L : \Ho_\gamma \to \Ho_\gamma$ defined by $L (\psi) = \psi \circ F - \zeta \psi$ is a linear map on the vector space $\Ho_\gamma$ and is invertible: if we order the monomials $x^{\Alp}$ which are a basis of $\Ho_\gamma$ according to our order $\prec_{\lambda}$ ($x^{\Alp} < x^{\Bet}$ if $\Alp \prec \Bet$), then the matrix representing $L$ is upper triangular with $\lambda^{\gamma} - \zeta$ on the diagonal. Moreover, the set $\Kdeal \cap \Ho_\gamma$ is a linear subspace of $\Ho_\gamma$ invariant under the action of $L$. Hence we get that its inverse stabilizes $\Kdeal \cap \Ho_\gamma$ too.
	\end{proof}
	
	\begin{proof}[Proof of proposition \ref{prop:IJ}]
		Let $\Jdeal'$ be the ideal generated by all the polynomials $\phi^{(i)}_\gamma$ for all $i \in \{1, \ldots, r\}$ and all $\gamma \in \Gamma$. Clearly $\Ideal \subseteq \Jdeal'$ and $\Jdeal \subseteq \Jdeal'$. We thus need to prove $\Jdeal' \subseteq \Jdeal$ which implies $\Ideal \subseteq \Jdeal$. In order to do so, we introduce for every $\gamma \in \Gamma$ and $i \in \{1, \ldots, r\}$ the set:
		$$
		S_{i,\gamma} = \{(j,\delta),  \delta \prec_{\lambda} \gamma \text{ or } \delta=\gamma \text{ and } j \leq i\}.
		$$
		
		We say that $(i,\gamma) \leq (j,\delta)$ if $S_{i,\gamma} \subseteq S_{j,\delta}$.
		
		We also define $\Jdeal_{i,\gamma} = \langle \phi^{(j)}_\delta, (j,\delta) \in S_{i,\gamma} \rangle$.
		
		We will prove by induction on $\gamma$ and $i$ that $\Jdeal_{i,\gamma}$ is invariant under the action of $F$ and that it is generated by the polynomials $P^{(j)} = \phi^{(j)}_{\gamma_j}$ such that $(j, \gamma_j) \in S_{i,\gamma}$.
		
		We first recall that we have for all $i \in \{1, \ldots, r\}$:
		
		\begin{align}
			\phi^{(i)} \circ F &= \sum_{\gamma \in \Gamma} \phi^{(i)}_\gamma \circ F\label{eq:somme}\\
			&=  A^i_1 \phi^{(1)} + \ldots + A^i_r \phi^{(r)}.\label{eq:decomp}
		\end{align}

		\underline{Base case}
		
		For all $i \in \{1, \ldots, r \}$ we have $\phi^{(i)} (0) = 0$ as $0 \in (X, 0)$, so we see that $\Jdeal_{1, 0} = \{0\}$ and the statement is true at the first step.

		\underline{Induction step}
		
		Let us now prove it for the rank $(i, \gamma)$. We suppose that the property is true for the preceding pair $(j,\delta)$ (which is equal to $(i-1, \gamma)$ if $i \neq 1$ and $(r,\gamma')$ with $\gamma'$ preceding $\gamma$ in $\Gamma$ if $i=1$).
		
		If we look at the parts of $\lambda$-degree $\gamma$ in \eqref{eq:somme} and \eqref{eq:decomp}, we get:	
		$$
		\phi^{(i)}_\gamma \circ F = \sum_{k=1}^r \sum_{\substack{\gamma', \gamma'' \in \Gamma \\ \gamma'+\gamma''=\gamma}} A^{i}_{k,\gamma'} \phi^{(k)}_{\gamma''} =  A^i_{i,0} \phi^{(i)}_\gamma + \sum_{(i', \gamma') < (i, \gamma)} A^{i}_{i', \gamma - \gamma'} \phi^{(i')}_{_\gamma'}
		$$
		where $\gamma - \gamma'$ is the equivalent class $\log_{\lambda} (\lambda^{\gamma}/\lambda^{\gamma'})$ if it exists. 
		The second equality can be written as all of the indexes in the first sum can be taken smaller or equal to $(i, \gamma)$. We indeed know that $\gamma'+\gamma''=\gamma$ implies that $\gamma'' \preceq \gamma$ with equality only if $\gamma' = 0$. But then $A^i_{j,0} \neq 0$ only if $j \leq i$ as $A_0$ is lower triangular.
		
		Therefore, all of the polynomials in the second sum are in $\Jdeal_{j,\delta}$. Hence
		\begin{equation}\label{eq:inI}
			\phi^{(i)}_\gamma \circ F - A^i_{i,0} \phi^{(i)}_\gamma \in \Jdeal_{j,\delta}.
		\end{equation}
		
		We then have two possibilities:
		
		- if $A^i_{i,0} \neq \lambda^{\gamma}$, by lemma \ref{lem:kdeal} we have that $\phi^{(i)}_\gamma \in \Jdeal_{j,\delta}$. Thus $\Jdeal_{i,\gamma} = \Jdeal_{j,\delta}$ and the properties are still verified.
		
		- if $A^i_{i,0} = \lambda^{\gamma}$, then $P^{(i)} = \phi^{(i)}_\gamma$. So $\Jdeal_{i,\gamma} = \Jdeal_{j,\delta} \oplus \langle P^{(i)} \rangle$. It is generated by some of the polynomials $P^{(j)}$ as $\Jdeal_{j,\delta}$ is, and \eqref{eq:inI} shows that it is invariant under the action of $F$.
		
		We have thus proved the property at rank $(i, \gamma)$, that is for all $(i,\gamma) \in \{1, \ldots, r\} \times \Gamma$,  $\Jdeal_{i,\gamma} \subseteq \Jdeal$.
	\end{proof}
	
	\begin{rem}\label{rem:ordcoord}
		As we have ordered the coordinates so that $\gamma_i \leq \gamma_j$ if $i < j$, we see that the polynomials $P^{(i)}$ are met ``in the right order'' in the proof: for all $i \in \{1, \ldots, r\}$ and $\gamma \in \Gamma$, we have:
		$$
		\Jdeal_{i,\gamma} = \langle P^{(j)} \rangle_{(j, \gamma_j) \leq (i, \gamma)}.
		$$
		As $\gamma_i \leq \gamma_j$ if $i < j$, it implies that for all $(i, \gamma)$, there exists $j \in \{0, \ldots, r\}$ such that $\Jdeal_{i,\gamma} = \langle P^{(1)}, \ldots, P^{(j)} \rangle$ and we thus get: $P^{(i)} \circ F \in \langle P^{(j)} \rangle_{j \leq i}$.
		
		Moreover, we also have that for all $i < j$, $\phi^{(i)}_{\gamma_j} \in \langle P^{(k)} \rangle_{(k, \gamma_k) \leq (i, \gamma_j)}$.
	\end{rem}

	\subsection{Second inclusion}\label{ssect:JI}
	
	We now only need to prove that the ideal generated by the $P^{(i)}$ is $\Ideal$ to finish the proof of theorem \ref{thm:InvId}.
	
	\begin{prop}\label{prop:JI}
		With the previous notations, we have $\Jdeal \subseteq \Ideal$.
	\end{prop}
	
	In order to prove this proposition, we pick a subset $E$ of $\{1, \ldots, r\}$ such that the family $P^{(i)}$ is a mimimal set of generators of $\Jdeal$ and such that $P^{(i)} \in \langle P^{(j)} \rangle_{j \leq i, j \in E}$. We then show that $\langle P^{(i)} \rangle_{i \in E} = \langle \phi^{(i)} \rangle_{i \in E}$ which concludes the proof.
	
	\begin{lem}\label{lem:E}
		There exists a subset $E \subseteq \{1, \ldots, r\}$ such that the family $(P^{(j)})_{j \in E}$ is a minimal set of generators of $\Jdeal$ and such that for all $i \in \{1, \ldots, r\}$:
		$$
		P^{(i)} \in \langle P^{(j)} \rangle_{j \leq i, j \in E}.
		$$
	\end{lem}
	
	\begin{proof}
		Set $\mathcal{E} = \left\{F \subset \{1, \ldots, r\}, (P^{(j)})_{j \in F} \text{ is a minimal set of generators of } \Jdeal \right\}$. The map $S : \mathcal{E} \to \ZZ$ defined by $\phi : F \mapsto \sum_{j \in F} j$ induces a partial order on $\mathcal{E}$, and let $E \in \mathcal{E}$ be minimal with respect to this partial order.
		
		Then $(P^{(j)})_{j \in E}$ generates $\Jdeal$ by definition, and we claim that it satisfies the second condition of the lemma.
		
		Let us assume that there exists $i \notin E$ such that $P^{(i)} \notin \langle P^{(j)} \rangle_{j \leq i, j \in E}$ (we assume $i \notin E$ as the second part of the lemma obviously holds for $i \in E$).
		
		As $\Jdeal = \langle P^{(j)} \rangle_{j \in E}$, we can decompose $P^{(i)} = \sum_{j \in E} u_j P^{(j)}$. We can first look at the parts of $\lambda$-degree $\gamma_i$ on both sides, and we still get a relation of the same form. Moreover the assumption shows that there exists $k > i, k \in E$ such that $u_k \neq 0$. However as $P^{(k)}$ has degree at least $\gamma_i$, the part of $\lambda$-degree $\gamma_i$ in $u_k P^{(k)}$ can only be $u_k (0) P^{(k)}$.
		
		But then we can write:
		$$
		P^{(k)} = \frac{1}{u_k} \left(P^{(i)} -  \sum_{j \in E, j \neq k} u_j P^{(j)}\right).
		$$
		
		The set $E \setminus \{k\} \cup \{i\}$ is thus also in $\mathcal{E}$ and is smaller than $E$ according to the partial order, which is a contradiction.
	\end{proof}
	
	This lemma enables us to keep the properties described in remark \ref{rem:ordcoord} even if we removed some coordinates.
	
	\begin{lem}\label{lem:still}
		Let $E \subseteq \{1, \ldots, r\}$ be a subset that satisfies the property of lemma \ref{lem:E}.
		
		Then we still have as in remark \ref{rem:ordcoord}:
		$$
		P^{(i)} \circ F \in \langle P^{(j)} \rangle_{j \leq i, j \in E}.
		$$
		for all $i \in \{1, \ldots, r\}$
		
		Moreover, we also have that for all $i < j$, $\phi^{(i)}_{\gamma_j} \in \langle P^{(k)} \rangle_{(k, \gamma_k) \leq (i, \gamma_j), k \in E}$.
	\end{lem}
	
	\begin{proof}
		We know that $P^{(i)} \circ F \in \langle P^{(j)} \rangle_{j \leq i}$. Moreover, for all $k \notin I$, lemma \ref{lem:E} shows that $P^{(k)} \in \langle P^{(j)} \rangle_{j \leq k, j \in E}$, which implies that $P^{(i)} \circ F \in \langle P^{(j)} \rangle_{j \leq i, j \in E}$.
		
		We also get from remark \ref{rem:ordcoord} that $\phi^{(i)}_{\gamma_j} \in \langle P^{(k)} \rangle_{(k, \gamma_k) \leq (i, \gamma_j)}$, so we can conclude similarly that $\phi^{(i)}_{\gamma_j} \in \langle P^{(k)} \rangle_{(k, \gamma_k) \leq (i, \gamma_j), k \in E}$.
	\end{proof}

	\begin{proof}[Proof of proposition \ref{prop:JI}]
		We can first pick a subset $E \subseteq \{1, \ldots, r\}$ that satisfies the conditions of lemma \ref{lem:E}.
		
		As we have already proved that $\Ideal \subseteq \Jdeal$, for all $i \in \{1, \ldots, r\}$ we can decompose $\phi^{(i)}$ in the following way:
		\begin{equation}\label{eq:B}
			\phi^{(i)} = \sum_{j \in E} B^{i}_{j} P^{(j)}
		\end{equation}
		with $B^{i}_{j}$ holomorphic functions.
		
		We claim that the square matrix $B = (B^{i}_{j})_{(i, j) \in E \times E}$ is invertible. In order to do so, we just need to prove that its determinant is invertible, which occurs if and only if its constant term is non zero. We can therefore only check if the matrix with the constant terms $B_0 = (B^{i}_{j} (0))_{(i, j) \in E \times E}$ is invertible.
		
		Let $i \in E$, if we look at the parts of $\lambda$-degree $\gamma_i$ in \eqref{eq:B}, we get:
		$$
		P^{(i)} = \left(\sum_{j \in E} B^{i}_{j} P^{(j)}\right)_{\gamma_i}.
		$$
		If $B^{i}_{i} (0) \neq 1$, then we can express $P^{(i)}$ in terms of the other polynomials $P^{(j)}, j \in E$ which would contradict the minimality of $E$, so $B^{i}_{i} (0) = 1$.
		
		Now let $i \in E$ and $j > i$, $j \in E$. If we look at the parts of $\lambda$-degree $\gamma_j$ in \eqref{eq:B}, we get:
		$$
		\phi^{(i)}_{\gamma_j} = \left(\sum_{k \in E} B^{i}_{k} P^{(k)}\right)_{\gamma_j}.
		$$
		
		However we know (cf. lemma \ref{lem:still}) that $\phi^{(i)}_{\gamma_j} \in \langle P^{(k)} \rangle_{(k, \gamma_k) \leq (i, \gamma_j), k \in E}$, so we can also decompose:
		$$
		\phi^{(i)}_{\gamma_j} = \sum_{k \in E, (k, \gamma_k) \leq (i, \gamma_j)} U^{i}_{k} P^{(k)}.
		$$
		
		We can thus write:
		$$
		B^{i}_{j} (0) P^{(j)} = \sum_{k \in E, (k, \gamma_k) \leq (i, \gamma_j)} U^{i}_{k} P^{(k)} - \left(\sum_{k \in E, k \neq j} B^{i}_{k} P^{(k)}\right)_{\gamma_j}.
		$$
		As $(j, \gamma_j) > (i, \gamma_j)$, we observe that $P^{(j)}$ does not appear on the right, so $B^{i}_{j} (0) = 0$ as the converse would contradict the minimality of $E$.
		
		We have thus proved that $B_0$ is a lower triangular matrix with $1$'s on the diagonal, which proves that it is invertible.
		
		Let $\beta = (\beta^i_j)_{(i, j) \in E \times E}$ be the inverse of $B$, we get:
		$$
		P^{(i)} = \sum_{j \in E} \beta^i_j \phi^{(j)}
		$$	
		for all $i \in E$.
		
		We then conclude that $\Jdeal = \langle P^{(i)} \rangle_{i \in E} = \langle \phi^{(i)} \rangle_{i \in E} \subseteq \Ideal$.
	\end{proof}
	
	\begin{rem}\label{rem:Etout}
		We have shown that $\Ideal$ is generated by the $P^{(i)}, i \in E$. As the minimal number of generators of $\Ideal$ is $r$, it proves that we actually have $E = \{1, \ldots, r\}$ and that $P^{(i)} \neq 0$ for all $i$ (which also implies that $\gamma_i \in \Gamma$ for all $i$).
	\end{rem}
	
	\begin{ex}\label{ex:figure}
		
		In figure \ref{fig:ill}, the point on the $\gamma$-th row and the $i$-th column represents the weighted homogeneous polynomial $\phi^{(i)}_{\gamma}$.
		
		In the proof of proposition \ref{prop:IJ}, we show that the ideal $\Jdeal_{2, \gamma_5}$ is generated by the polynomials $P^{(i)}$ that are ``in the ideal'', that is $P^{(1)}$, $P^{(2)}$, $P^{(3)}$ and $P^{(4)}$.
		
		The lemmas \ref{lem:E} and \ref{lem:still} show that we can pick a minimal set of generators while still preserving the property that $\Jdeal_{2, \gamma_5}$ is generated by some of the $P^{(i)}$'s that are contained in it.
		
		Finally, in the proof of proposition \ref{prop:JI}, we show that this implies that we do not need the other polynomials ($P^{(5)}$ and higher) to decompose $\phi^{(2)}_{\gamma_5}$. If we write:
		$$
		\phi^{(2)} = \sum_{j \in E} B^{2}_{j} P^{(j)}
		$$
		as in \eqref{eq:B}, this thus shows that $B^{2}_5 (0) = 0$ when we compare the parts of $\lambda$-degree $\gamma_5$.
	\end{ex}

	\begin{figure}[hb]
		\def\svgwidth{0.6\columnwidth}
		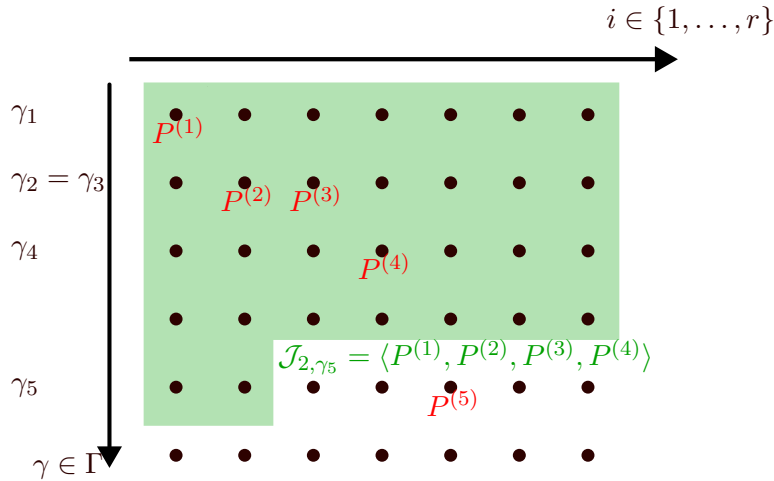
		
		\caption{Illustration of the proof of prop. \ref{prop:JI} (see details in ex \ref{ex:figure}).}\label{fig:ill}
	\end{figure}

	\bibliography{../../bibliographie}
	\bibliographystyle{plain}
	
\end{document}